\documentclass[11pt]{amsart}
\usepackage[english]{babel}
\usepackage[latin1]{inputenc}
\usepackage[dvips,final]{graphics}
\usepackage{amsmath,amsfonts,amssymb,amsthm,amscd,array,stmaryrd,mathrsfs}
\usepackage{pstricks}
 \usepackage[all]{xy}
\usepackage{graphicx}
\usepackage{color}              % Need the color package
\usepackage{epsfig}
\usepackage{psfrag}
\usepackage{epstopdf}

\usepackage{textcomp}

\let\ssection=\section
\renewcommand{\section}{\setcounter{equation}{0}\ssection}

\theoremstyle{plain}
\newtheorem{thm}{Theorem}%[section]
\newtheorem{lem}{Lemma}[section]

\newtheorem{prop}[lem]{Proposition}

\theoremstyle{definition}
\newtheorem{defi}[lem]{Definition}
\newtheorem{rem}[lem]{Remark}
\newtheorem{ex}[lem]{Example}

\newcommand{\R}{\mathbb{R}}
\newcommand{\Z}{\mathbb{Z}}
\newcommand{\C}{\mathbb{C}}

\newcommand{\bbH}{\mathbb{H}}

\newcommand{\bbO}{\mathbb{O}}

\newcommand{\Zn}{(\Z/2\Z)^n}

\newcommand{\cN}{\mathcal{N}}

\newcommand{\Span}{\mathrm{Span}}
\newcommand{\Cl}{\mathit{C}\ell}

\newcommand{\gS}{\mathfrak{S}}

\newcommand{\card}{\textup{card}}
\newcommand{\Id}{\textup{Id}}

\def\a{\alpha}

\def\d{\delta}

\def\g{\gamma}

\def\r{\rho}

\def\vfi{\varphi}

\begin{document}

\title[New solutions to the Hurwitz problem]{New solutions to the Hurwitz problem\\ on square identities}

\author{
Anna Lenzhen,
Sophie Morier-Genoud
\and
Valentin Ovsienko}

\address{
Anna Lenzhen,
Universit\'e de Rennes 1,
Campus de Beaulieu, 35042 Rennes Cedex France}
\email{anna.lenzhen@univ-rennes1.fr}

\address{Sophie Morier-Genoud,
Institut de Math\'ematiques de Jussieu
UMR 7586
Universit\'e Pierre et Marie Curie,
4 place Jussieu, case 247,
75252 Paris Cedex 05
}
\email{sophiemg@math.jussieu.fr}

\address{
Valentin Ovsienko,
CNRS,
Institut Camille Jordan,
Universit\'e Claude Bernard Lyon~1,
43 boulevard du 11 novembre 1918,
69622 Villeurbanne cedex,
France}

\email{ovsienko@math.univ-lyon1.fr}

\date{}

\begin{abstract}
The Hurwitz problem of composition of quadratic forms,
or of ``sum of squares identity'' is tackled with the help of a
particular class of $\Zn$-graded non-associative algebras
generalizing the octonions.
This method provides explicit formulas for the
classical Hurwitz-Radon identities and leads to new solutions
in a neighborhood of the Hurwitz-Radon identities.
\end{abstract}

\maketitle

\thispagestyle{empty}

%\vspace{-1cm}

%\newpage

%\tableofcontents

%%%%%%%%%%%%%%%%%%%%%%%%%%%%%%%%%
\section{Introduction}
%%%%%%%%%%%%%%%%%%%%%%%%%%%%%%%%%

A square identity of size $[r,s,N]$ is an identity
\begin{equation}
\label{rsN}
(a_1^2+\cdots{}+a_r^2)\,(b_1^2+\cdots{}+b_s^2)
=c_1^2+\cdots{}+c_N^2,
\end{equation}
where 
$a_1,\ldots,a_r$ and $b_1,\ldots,b_s$ are independent indeterminates and each
$c_i$ is a bilinear form in  $(a_1,\ldots,a_r)$ and $(b_1,\ldots,b_s)$ with integer coefficients.
The problem to determine all triples of positive integers $r,s,N$
for which there exists a square identity of size $[r,s,N]$ was formulated by Hurwitz \cite{Hur1}
and remains widely open.
Such triples are called \textit{admissible}.
An admissible triple $[r,s,N]$ is \textit{optimal} if $r$ and~$s$
cannot be increased and $N$ cannot be decreased.
This problem has  various applications and interpretations in  
number theory, linear algebra, geometry and topology, see \cite{Sha},\cite{Squa} for a complete overview.

The simplest example is called Brahmagupta's two-square identity
$$
(a_1^2+a_2^2)\,(b_1^2+b_2^2)=
(a_1b_1-a_2b_2)^2+(a_1b_2+a_2b_1)^2,
$$
that can be easily proved algebraically, but can also be understood as equality between 
the products of the norms of two complex numbers
and the norm of their product:
\begin{equation}
\nonumber
\cN(a)\,\cN(b)=\cN(a\cdot b).
\end{equation}
The similar equalities in the algebras of quaternions $\bbH$ and of octonions $\bbO$
coincide with Euler's four-square identity (1748)
and Degen's eight-square identity (1818), respectively.
We refer to \cite{ConSmi} for the history of these identities and their meaning.

In 1898, Hurwitz proved his celebrated theorem \cite{Hur1} stating that $N$-square identities
exist only for $N=1,2,4$ and $8$.
Actually, Hurwitz proved a stronger ``non-existence'' result in the case where
$c_i$ are bilinear forms with complex coefficients.
In 1918 (published in 1922), Hurwitz \cite{Hur} generalized his theorem
and established the largest integer $r=\r(N)$ 
for which there exists a $[r,N,N]$-identity with complex coefficients. 
In 1922, Radon  \cite{Rad} independently solved a similar problem in the case of real coefficients.
The result is the same in both cases:
write $N$ in a form $N=2^n(2k+1)$, then
$$
\r(N)=\left\{ \begin{array}{lcll}
2n+1, \quad&n\equiv& 0 &\mod 4\\
2n, \quad&n\equiv& 1,2 &\mod 4\\
2n+2, & n\equiv &3 &\mod 4.
\end{array}
\right.
$$
The function $\r$ is called the \textit{Hurwitz-Radon function}
(this function also appears in topology~\cite{Ada} and linear algebra \cite{ALP}).
Note that, Hurwitz's constructive proof leads to $[r,N,N]$-identities with 
integer coefficients.
According to \cite{Gab},
this even implies the existence of such identities with 
coefficients in $\{-1,0,1\}$.

A nice feature of the Hurwitz-Radon identities is
their relation to Clifford algebras.
Every square identity of size $[r+1,N,N]$ defines a representation of
the real Clifford algebra $\Cl_{0,r}$ by $N\times{}N$-real matrices.
In this representation the matrices corresponding to the generators of $\Cl_{0,r}$
are orthogonal and skew-symmetric matrices with integer coefficients.
Conversely, every such representation of $\Cl_{0,r}$ defines an
$[r+1,N,N]$-identity.
In his original work \cite{Hur1}, Hurwitz implicitly constructed Clifford algebras
and their representations.
However, this was formulated in an explicit and conceptual way much later
(see \cite{Sha} for more details).

Besides the Hurwitz-Radon formula, a number of solutions to the Hurwitz problem
are known.
We refer to  \cite{Ade},\cite{Lam1},\cite{Lam2},\cite{LY},\cite{Sanch},\cite{Yiu},
for concrete examples of admissible triples $[r,s,N]$
and various methods to prove their optimality.
The complete list of the admissible triples with $N\leq32$ is given in \cite{Yiu1}.

In \cite{Yuz} and \cite{Yuz2}, Yuzvinsky introduced a new approach
to construct square identities, see also \cite{LS}.
His method is based on \textit{monomial pairings} on the group algebra over $\Zn$.
In particular,
existence of infinite series of identities of size
\begin{equation}
\label{YuzEq}
[2n+2, 2^n-\vfi(n),2^n],
\end{equation}
 where the function $\vfi(n)$ is given by the binomial coefficients
\begin{equation}
\nonumber
\vfi(n)={n\choose n/2},\;
n\equiv 0\mod 4,
\qquad
\vfi(n)=2\,{n-1\choose (n-1)/2},\;
n\equiv 1 \mod 4,
\end{equation}
was established.
The above values of $[r,s,N]$ are called Yuzvinsky-Lam-Smith formulas (see \cite{Sha}).
Another series of triples:
\begin{equation}
\label{Angel}
\left[2n\;,\; 2^n-2n\;, \;2^n-2\right]
\end{equation}
can be found in \cite{Ang}.
To the best of our knowledge, the Hurwitz-Radon formula,
the expressions (\ref{YuzEq}) and (\ref{Angel}) as well as their direct consequences, are
the only known infinite series of solutions to the Hurwitz problem.

In the present paper, we obtain infinite families of admissible
triples $[r,s,N]$  similar to~\eqref{YuzEq}. 
Our method is similar to the Yuzvinsky method, but the monomial pairings
we consider are different.
We use a series of algebras constructed in the recent work \cite{MGO}.
These algebras generalize the classical algebra of octonions  and 
Clifford algebras at the same time. 
As an application of these algebras, explicit expressions
for the  Hurwitz-Radon identities were provided in~\cite{MGO}. Here we extend the method to 
obtain more general identities. 
As well as  \eqref{YuzEq}, the triples that we obtain are close to the Hurwitz-Radon triple.
Whereas \eqref{YuzEq} is obtained from
$[\r(2^n),\,2^n,\,2^n]$ by increasing the first entry, $\r(2^n)$,
we will decrease the third entry, $2^n$.
Let us also mention that 
our formulas include \eqref{Angel} as the simplest case.

Let us also mention that optimality of  $[r,s,N]$ is often a difficult problem.
For instance, it is not known whether  (\ref{YuzEq}) is optimal.
We show that our formulas generate several optimal (and several 
best known) values $[r,s,N]$.
However, unfortunately we do not have any result on optimality of our formulas.

The paper is organized as follows.
In Section \ref{MainRS}, we formulate the main results.
In Section~\ref{TwiS}, we introduce the general notion of twisted group algebra
over the abelian group~$\Zn$.
In Section \ref{NeshAlgSec}, we recall the definition of the algebras $\bbO_n$
and introduce the notion of Hurwitzian subset of $\Zn$.
Proofs of the main results are given in Section \ref{proofSec}.

%%%%%%%%%%%%%%%%%%%%%%%%%%%%%%%%%
%%%%%%%%%%%%%%%%%%%%%%%%%%%%%%%%%
\section{The main results}\label{MainRS}
%%%%%%%%%%%%%%%%%%%%%%%%%%%%%%%%%
%%%%%%%%%%%%%%%%%%%%%%%%%%%%%%%%%

The first family of admissible triples that we obtain
depends on two parameters, $\ell$ and~$k$.

\begin{thm}
\label{FirstMain}
(i)
For every $n$, there exist square identities of size $[r,s,N]$ with
\begin{equation}
\label{step}
\left\{
\begin{array}{rcl}
r&=&2n,\\[6pt]
s&=&2^n-2\left({k-1\choose2}+\ell+1\right)n +4\,{k\choose3}+2\,k\,\ell,\\[12pt]
N&=&2^n-2\left({k-1\choose2}+\ell+1\right),
\end{array}
\right.
\end{equation}
for all $1\leq\ell<k\leq{}n$.

(ii) If $n\equiv3\mod\,4$, then there are also identities of size
$[r+2,s-2,N]$, where $r,s,N$ are as in \text{\eqref{step}}.
\end{thm}

For small values of $\ell$ and $k$, formula  \eqref{step} gives the following triples:
$$
\left[2n,\,2^n-4n+4,\,2^n-4\right],
\qquad
\left[2n,\,2^n-6n+10,\,2^n-6 \right],
\qquad
\left[2n,\,2^n-8n+16,\,2^n-8 \right].
$$
We conjecture that the values (\ref{step}) are optimal if $n$ is sufficiently large.

Let us give a few examples of concrete numeric values.

\begin{ex}
{\rm
(a)
For $n=5$, we obtain the following admissible triples:
$[10,12,26]$, $[10,16,28]$;
the latter is optimal (as proved in \cite{Yiu1}).

(b)
For $n=6$, we obtain $[12,38,58],\;[12,44,60]$ that correspond to the best known values,
(cf,~\cite{Sha}, pp.292--293).

(c)
For $n=7$, we obtain the values $[14,88,120],\;[14,96,122],\;[14,104,124]$ from (\ref{step}).
Furthermore, part (ii) of the theorem leads to 
$[16,86,120]$, $[16,94,122]$, $[16,102,124]$.

(d)
For $n=8$, we obtain
$[16,218,250]$, $[16,228,252]$.
}
\end{ex}

The second family of identities is parametrized by
an integer $k<\frac{n}{2}$.
The values of $s$ and $N$ are given in terms of sums of binomial coefficients.
\goodbreak

\begin{thm}
\label{SecMain}
(i)
If $n\equiv 1 \mod 4$, then there exist square identities of size $[r,s,N]$,
where
\begin{equation}
\label{stepDva}
r=2n,
\qquad
s= \sum_{0\leq i \leq k-1}2\,{n\choose m-i},
\qquad
N=\sum_{0\leq i \leq k}2\,{n\choose m-i},
\end{equation}
where $m=(n-1)/2$ and $1\leq{}k\leq{}m$.

(ii)
If $n\equiv 3 \mod 4$, then there exist square identities of size $[r+2,s,N]$,
where $r,s$ and~$N$ are as in \text{\eqref{stepDva}}.
\end{thm}

Let us give the first and simplest example.
In the case $k=1$, the formula (\ref{stepDva}) reads
$$
\Big[ 2n\; ,\; 2{n\choose m}\; ,\; 2{{n+1}\choose m}\Big]\;,
\qquad
\Big[ 2n+2\; ,\; 2{n\choose m}\; ,\; 2{{n+1}\choose m}\Big],
$$
for $n\equiv 1 \mod 4$ and $n\equiv 3 \mod 4$, respectively.

\begin{ex}
{\rm
We obtain the values $[10,20,30]\;,\; [16,70,112]\;, [18,252,420]$.
The second one strengthens the
known triple $[16,64,112]$, see~\cite{Sha}, p.294.
}
\end{ex}

%%%%%%%%%%%%%%%%%%%%%%%%%%%%%%%%%
%%%%%%%%%%%%%%%%%%%%%%%%%%%%%%%%%
\section{Twisted group algebras over $(\Z/2\Z)^n$}\label{TwiS}
%%%%%%%%%%%%%%%%%%%%%%%%%%%%%%%%%
%%%%%%%%%%%%%%%%%%%%%%%%%%%%%%%%%

In this section, we recall the definition of twisted group algebras
over the abelian group~$\Zn$.
These algebras include  Clifford algebras and the algebra of octonions.
We then explain Yuzvinsky's idea \cite{Yuz} of construction of square identities
out of a twisted group algebra.

%%%%%%%%%%%%%%%%%%%%%%%%%%%%%%%%%
\subsection{Definition and examples of twisted group algebras}
%%%%%%%%%%%%%%%%%%%%%%%%%%%%%%%%%

Consider the group algebra $\R\left[(\Z/2\Z)^n\right]$ over the
abelian group $\Zn$.
Denote by $u_x$ the natural basis vectors indexed by elements of the group $x\in\Zn$.
One then has:
$$
\R\left[(\Z/2\Z)^n\right]=\bigoplus_{x\in \Zn }\R\, u_x.
$$
Given a function $f:\Zn\times \Zn\to\Z/2\Z$, satisfying $f(x,0)=f(0,x)$ for all $x\in \Zn$, 
one defines a twisted product on $\R\left[(\Z/2\Z)^n\right]$
$$
u_x\cdot{}u_y=\left(-1\right)^{f(x,y)}
u_{x+y},
$$
for all $x,y\in\left(\Z/2\Z\right)^n$.
The defined algebra is called
a \textit{twisted group algebra} over $(\Z/2\Z)^n$
and denote by $(\R\left[(\Z/2\Z)^n\right],f)$.
The basis vector $u_0$ is a unit in $(\R\left[(\Z/2\Z)^n\right],f)$.

\begin{ex}
{\rm
Recall that the Clifford algebra $\Cl_{0,n}$
is the associative algebra with $n$ generators
$\g_1,\ldots,\g_n$ and relations
$$
\g_i^2=-1,
\qquad
\g_i\g_j=-\g_j\g_i.
$$
Write elements of $\Zn$ in the form
$x=(x_1,\ldots,x_n)$, where $x_i=0,1$, and choose the function
\begin{equation}
\label{Cliffunet}
f_{\Cl}(x,y)=\sum_{1\leq{}i \leq{}j\leq{}n}x_iy_j.
\end{equation}
It turns out that the twisted group algebra $(\R\left[(\Z/2\Z)^n\right],f_{\Cl})$ is isomorphic to
the Clifford algebra $\Cl_{0,n}$, see \cite{AM} for the details.
For instance, if $n=2$, then this is nothing but the algebra of quaternions $\bbH$.
}
\end{ex}

The algebra $(\R\left[(\Z/2\Z)^n\right],f)$ is associative if and only if
$f$ is a 2-cocycle on the group~$\Zn$, i.e.,
$$
f(y,z)+f(x+y,z)+f(x,y+z)+f(x,y)=0,
$$
for all $x,y,z\in\Zn$.
This is obviously the case for the bilinear function (\ref{Cliffunet}).

The following example realizes the classical non-associative
algebra of octonions $\bbO$ as a twisted group algebra
with a cubic twisting function~$f$ (that is not a 2-cocycle).
This result was found in \cite{AM1}.

\begin{ex}
\label{ExO}
{\rm
The twisted group algebra $(\R\left[(\Z/2\Z)^3\right],f)$ defined by the function
$$
f(x,y)=
x_1x_2y_3+x_1y_2x_3+y_1x_2x_3+
\sum_{1\leq{}i\leq{}j\leq3}\,x_iy_j
$$
is isomorphic to the algebra of octonions $\bbO$.
}
\end{ex}

%%%%%%%%%%%%%%%%%%%%%%%%%%%%%%%%%
\subsection{Multiplicativity criterion}\label{MultPaS}
%%%%%%%%%%%%%%%%%%%%%%%%%%%%%%%%%

In this subsection, we explain the Yuzvinsky original idea (see~\cite{Yuz} and also \cite{Yuz2}).

Let us consider a twisted group algebra $(\R\left[(\Z/2\Z)^n\right],f)$.
We define the Euclidean norm:
\begin{equation}
\label{EuclEq}
\cN(a)=\sum_{x\in(\Z/2\Z)^n}a_x^2,
\end{equation}
for $a=\sum_{x\in \Zn}a_x\,u_x$, where $a_x\in \R$.

Recall that a composition algebra is a normed algebra such that, for any two elements $a$ and $b$, 
\begin{equation}
\label{NormProd}
\cN(a)\,\cN(b)=\cN(a\cdot b).
\end{equation}

It is well-known (and follows for instance from Hurwitz's theorem)
that there are no composition algebras of dimension greater than $8$.
Therefore, if $n\geq4$, then the condition (\ref{NormProd}) cannot be satisfied for
arbitrary elements $a$ and~$b$
of the twisted group algebra $(\R\left[(\Z/2\Z)^n\right],f)$.
In order to find square identities,
we will look for subsets $A,B$ of $\Zn$, such that the condition~(\ref{NormProd}) 
is satisfied for 
$$
a\in\Span_\R(u_x,\,x\in{}A),
\qquad
b\in\Span_\R(u_x,\,x\in{}B).
$$

\begin{defi}
{\rm
Given two subsets $A,B\subset\Zn$, we say that $(A,B)$ is a 
\textit{multiplicative pair} if the condition (\ref{NormProd}) holds for any two elements of the form
$$
a=\sum_{x\in A}a_xu_x,
\qquad
b=\sum_{y\in B}b_yu_y.
$$
}
\end{defi}

Note that this definition depends on the twisting function $f$.
Moreover, existence of multiplicative pairs is an important property of $f$.

Recall that the \textit{sumset}, $A+B$, of subsets $A$ and $B$ 
in $(\Z/2\Z)^n$ is  the set formed by pairwise sums of the elements
of $A$ and~$B$, i.e. 
$$
A+B=\{a+b \;|\, a\in A,\, b\in B\}.
$$
\goodbreak

\begin{prop}
\label{SumProp}
If $(A,B)$ is a multiplicative pair, then there exists a square identity of size
$$
[r,s,N]=[\card{}(A),\,\card{}(B),\,\card{}(A+B)],
$$
which can be written explicitly as
$$
\Big(\sum_{x\in A}a_x^2\Big) \Big(\sum_{y\in B}b_y^2\Big)=\Big(\sum_{z\in A+B}c_z^2\Big)  
$$
where $$c_{z}=\sum\limits_{\substack{ (x,y)\in A\times B\\ x+y=z}}(-1)^{f(x,y)}\,a_x\,b_y.$$
\end{prop}

\begin{proof}
This is an immediate consequence of the property \eqref{NormProd}.
\end{proof}

The following important statement is proved in \cite{Yuz}.
It provides a criterion for subsets $A$ and $B$ to
be a multiplicative pair and, therefore, it also proves the existence of
corresponding square identities.
The proof is elementary and we give it here for the sake of completeness.

\begin{prop}
\label{propmult}
Let $A$ and $B$ be two subsets of $\Zn$. The following two conditions are equivalent
\begin{enumerate}
\item[(i)]
The pair $(A,B)$ is multiplicative.
\item[(ii)] 
For all $x\not=z\in A$ and $y\not=t\in B$ such that
$x+y+z+t=0$, one has 
\begin{equation}
\label{fcond}
f(x,y)+f(z,t)+f(x,t)+f(z,y)=1.
\end{equation}
\end{enumerate}
\end{prop}

\begin{proof}
The product of the Euclidean norms is obviously
$
\cN(a)\,\cN(b)=\sum_{x,y}\,a_x^2\,b_y^2.
$
On the other hand, one obtains:
$$
\cN(a\cdot{}b)=
\sum_{x+y+z+t=0}(-1)^{f(x,y)+f(z,t)}\,a_x\,b_y\,a_z\,b_t.
$$
It follows that the condition (\ref{NormProd}) is satisfied if and only if
the terms $a_x\,b_y\,a_z\,b_t$ and $a_x\,b_t\,a_z\,b_y$
cancel whenever $(x,y)\not=(z,t)$ and $a_x\,b_y\,a_z\,b_t\not=0$,
in other words, the corresponding signs are opposite.
Hence the result.
\end{proof}

Our next goal is to find a twisted group algebra $(\R\left[(\Z/2\Z)^n\right],f)$
that admits multiplicative pairs of large cardinality. 

%%%%%%%%%%%%%%%%%%%%%%%%%%%%%%%%%
%%%%%%%%%%%%%%%%%%%%%%%%%%%%%%%%%
\section{The algebras $\bbO_n$}\label{NeshAlgSec}
%%%%%%%%%%%%%%%%%%%%%%%%%%%%%%%%%
%%%%%%%%%%%%%%%%%%%%%%%%%%%%%%%%%

%The series of algebras $\bbO_n$ was introduced in \cite{MGO}.
%In this section, we recall their definition and the characterizing properties.
In this section, we apply the results of Section \ref{TwiS} to a particular series of twisted group algebras.
%%%%%%%%%%%%%%%%%%%%%%%%%%%%%%%%%
\subsection{Definitions}
%%%%%%%%%%%%%%%%%%%%%%%%%%%%%%%%%

We recall the definition and the main properties of the algebras $\bbO_n$
generalizing the algebra of octonions.
This section presents a brief account, the details can be found in \cite{MGO}.

\begin{defi}
{\rm
The algebra $\bbO_n$ is the twisted group algebra
defined by the function
\begin{equation}
\label{NashProd}
f_\bbO(x,y)=\sum_{1\leq{}i<j<k\leq{}n}
\left(
x_ix_jy_k+x_iy_jx_k+y_ix_jx_k
\right)+
\sum_{1\leq{}i\leq{}j\leq{}n}\,x_iy_j,
\end{equation}
for all $x=(x_1, \ldots, x_n)$ and $y=(y_1, \ldots, y_n)$,  elements of $\Zn$.
}
\end{defi}
Note that $\bbO_3$ is exactly the algebra of octonions $\bbO$, cf. Example \ref{ExO}.

The main property of the algebra $\bbO_n$ is that its structure is determined
(up to isomorphism) by 
a single function in one variable $$\a:\Zn \rightarrow\Z/2\Z$$
that we call the \textit{generating function}.

\begin{defi}
{\rm
Given a twisted group algebra $(\R\left[(\Z/2\Z)^n\right],f)$, we say that
this algebra has a generating function if there exists a function
$
\a:\Zn \rightarrow \Z/2\Z
$
such that
\begin{enumerate}
\item[(i)] 
For all $x,y\in \Zn$, the following equation is satisfied:
$$
f(x,y)+f(y,x)=\a(x+y)+\a(x)+\a(y),
$$
\item[(ii)] 
For all $x,y,z\in \Zn$,
\begin{equation*}
\label{Genalp2}
\begin{array}{rl}
&f(y,z)+f(x+y,z)+f(x,y+z)+f(x,y)\\[4pt]
&=\a(x+y+z)+\a(x+y)+\a(x+z)+\a(y+z)+\a(x)+\a(y)+\a(z).
\end{array}
\end{equation*}

\end{enumerate}
}
\end{defi}

\begin{rem}
{\rm
The above conditions (i) and (ii) have cohomological meaning.
In particular the expression on the right-hand side of (i)
and the left-hand side of (ii) are the differentials
$\d\a$ and $\d{}f$, respectively.
The idea of existence of a generating function
goes back to Eilenberg and MacLane (see \cite{Mac});
this idea is also crucial for the theory of error-correcting codes and code loops
(see \cite{Gri}).
}
\end{rem}

One can immediately check the following statement.

\begin{prop} \cite{MGO}
\label{propalpha}
(i)
The function
$$
\a_{\bbO}(x):=f_{\bbO}(x,x),
$$
for $x\in \Zn$, is a generating function for the algebra $\bbO_n$.

(ii) 
For $f=f_{\bbO}$ and $x+y+z+t=0$, one has
\begin{equation}
\label{GenFuncMult}
f_{\bbO}(x,y)+f_{\bbO}(z,t)+f_{\bbO}(x,t)+f_{\bbO}(z,y)
=\a_{\bbO}(x+z).
\end{equation}

(iii)
The function $\a_{\bbO}(x)$ depends only on the weight of the element $x$,
i.e., on the integer
$$
|x|=\sum_{1\leq i \leq n} x_i.
$$
\end{prop}

\noindent
Note that equation \eqref{GenFuncMult} is equivalent to  (\ref{fcond}).

The algebra $\bbO_n$ is the unique twisted group algebra
over $\Zn$ that admits a generating function
$\a_{\bbO}$ which is invariant under the action
of the group of permutations~$\gS_n$ on $\Zn$,
i.e., the action permuting the coordinates of $x$.

By (iii) of Proposition \ref{propalpha}, the value $\a_{\bbO}(x)$ only depends on the weight $|x|$. 
It turns out that $\a_{\bbO}$ is 4-periodic. 

\begin{equation*}
\label{GenFuncTab}
\a_{\bbO}(x)=\left\{
\begin{array}{ll}
0\;,& \text{ if } \;|x|\equiv 0\mod 4,\\[4pt]
1\;,&  \text{ otherwise.} 
\end{array}
\right.
\end{equation*}

%%%%%%%%%%%%%%%%%%%%%%%%%%%%%%%%%
\subsection{The Hurwitzian sets}\label{HurSecT}
%%%%%%%%%%%%%%%%%%%%%%%%%%%%%%%%%

We use the function $\a_{\bbO}$ to give a combinatorial  characterization 
of multiplicative pairs for the
algebra~$\bbO_n$.

\begin{lem}
A pair $(A,B)$ of subsets of $\Zn$ is
a multiplicative pair if and only if the following condition is satisfied.
If
$$
w\in\left(A+A
\right)\cap\left(B+B
\right)
$$
and $w\not=0$, then $|w|\not\equiv0\,(\!\!\mod4)$.
\end{lem}

\begin{proof}
The condition $0\not=w\in\left(A+A\right)$ means $w=x+z$
for some elements $x,z\in A$ such that $x\not=z$.
Similarly, $w=y+t$, where $y, t\in B$.
By \eqref{GenFuncMult}  and \eqref{fcond}, multiplicativity of a pair $(A,B)$  is equivalent to
\begin{equation}
\label{GenFuncFor}
\a_{\bbO}(x+z)=\a_{\bbO}(y+t)=1,
\end{equation}
for all $x,z\in A$ and $y,t\in B$ such that $x+z=y+t$.
Since $\a_{\bbO}$ is 0 only for elements of weight multiple of 4, the lemma now follows.
\end{proof}

\begin{defi}
{\rm
A subset $H\subset\Zn$ is called a \textit{Hurwitzian set} if
$H$ forms a multiplicative pair with $\Zn$, and $\card(H)$ is the greatest possible with
this condition.
}
\end{defi}
Let us fix the following notation for
some particular elements of $\Zn$:
 \begin{equation*}
\begin{array}{rcl}
e_0&:=& (0,0,\ldots,0),\\[4pt]
\overline{e_0}&:=& (1,1,\ldots,1),\\[4pt]
e_i&:=&(0,\ldots,0,1,0,\ldots,0), \text{ where 1 occurs at the $i$-th position,}\\[4pt]
\overline{e_i}&:=&(1,\ldots,1,0,1,\ldots,1), \text{ where 0 occurs at the $i$-th position}.
\end{array}
\end{equation*}

The following table provides a possible choice of
Hurwitzian sets and the size of the identities obtained 
by applying Proposition \ref{SumProp}, depending on the class of $n\mod4$.
\begin{equation}
\label{SetH}
 \setlength{\extrarowheight}{3pt}
 \begin{array}{c|l|l|l|l|}
&\text{Hurwitzian set}& \text{Size of identity}\\[6pt]
\hline
n\equiv 0 \mod 4& H=\{e_i, e_1+e_i\}&[2n, 2^n, 2^n]\\[6pt]
\hline 
n\equiv 1 \mod 4& H=\{ e_i, \overline{e_i}\} &[2n, 2^n, 2^n] \\[6pt]
\hline
n\equiv 2 \mod 4& H=\{e_i, e_1+e_i\}&[2n, 2^n, 2^n]\\[6pt]
\hline
n\equiv 3 \mod 4
&H=\{e_0, \overline{e_0},\, e_i, \overline{e_i}\}&[2n+2, 2^n, 2^n]\\[6pt]
\hline
\end{array}
\end{equation}
where $1\leq i \leq n$.
Indeed, it is easy to check that $|x+z|$ is never a multiple of 4, 
when~$x,z\in{}H$.
Therefore the condition (\ref{GenFuncFor}) is always satisfied.

\begin{rem}
{\rm
 For $n=1,2$ or $3\mod4$, the triples in the table are the
optimal triples given by Hurwitz-Radon theorem. When $n\equiv0\mod4$, 
we do not achieve the optimal $[2n+1, 2^n, 2^n]$. 
This is related to the fact that for $n\equiv0\mod4$, the algebra~$\bbO_n$ is not simple.
}
\end{rem}

%%%%%%%%%%%%%%%%%%%%%%%%%%%%%%%%%
%%%%%%%%%%%%%%%%%%%%%%%%%%%%%%%%%
\section{Generalized  Hurwitz-Radon identities}\label{proofSec}
%%%%%%%%%%%%%%%%%%%%%%%%%%%%%%%%%
%%%%%%%%%%%%%%%%%%%%%%%%%%%%%%%%%

In this section, we prove Theorems \ref{FirstMain} and \ref{SecMain}. The proofs are based on 
Proposition \ref{SumProp}.
The idea is to start with a multiplicative
pair $(H,B)$, where $B=\Zn$, and reduce the size of $H+B$ by removing elements from $B$.
We first give some examples. 

%%%%%%%%%%%%%%%%%%%%%%%%%%%%%%%%%
\subsection{The simplest examples}\label{SimExSec}
%%%%%%%%%%%%%%%%%%%%%%%%%%%%%%%%%

Let $H$ be a Hurwitzian set as in table \eqref{SetH}.

1) For $B=\Zn\setminus{}H$,
the sumsets are as follows:
$$
H+B=\Zn \setminus  \{e_0, e_1\},
\qquad
H+B=\Zn \setminus  \{e_0, \bar{e}_0\},
$$
for $n\equiv0,2\mod4$ and $n\equiv1,3\mod4$, respectively.
This leads to the formula \eqref{Angel}.
Furthermore, in the case  $n\equiv 3 \mod 4$, one obtains an identity of size
$$
[2n+2\;,\; 2^n-2n-2\;,\;  2^n-2].
$$

2) Assume $n$ is odd and let
$$
B=\Zn\setminus \left\{x\;,\; |x|=1,\,3,\,n-3,\,n-1 \right\}.
$$ 
The sumset is then of the form:
$$
H+B=\Zn\setminus \left\{x\;,\; |x|=0,\,2,\,n-2,\,n \right\}.
$$ 
This leads to identities of size
$$
\left[2n\;,\;2^n-2 \Big(n+ {n\choose 3}\Big) \;, \; 2^n- 2 {n\choose 2}-2\right],
$$
which corresponds to the second ``border case'' in Theorem \ref{FirstMain},
namely to formula (\ref{step}) with $(\ell,k)=(n-1,n)$.

3) Again, let $n$ be odd and write $n=2m+1$.
Set  $B=\left\{x\;,\; |x|=m,\,m+1\right\}.$ 
Then 
$H+B=\left\{x\;,\; |x|=m-1,\, m,\,m+1,\,m+2 \right\}$,
see Figure \ref{FirstF}. 
This leads to identities of size
$$
\left[ 2n\; ,\; 2{n\choose m}\; ,\; 2{{n+1}\choose m} \right]\;,\qquad
\left[ 2n+2\; ,\; 2{n\choose m}\; ,\; 2{{n+1}\choose m} \right],
$$
if $n\equiv 1 \mod 4$ and   $n\equiv 3 \mod 4$, respectively.
This is Theorem \ref{SecMain}, in the particular case $k=1$.
\begin{figure}[hbtp]
\includegraphics[width=9cm]{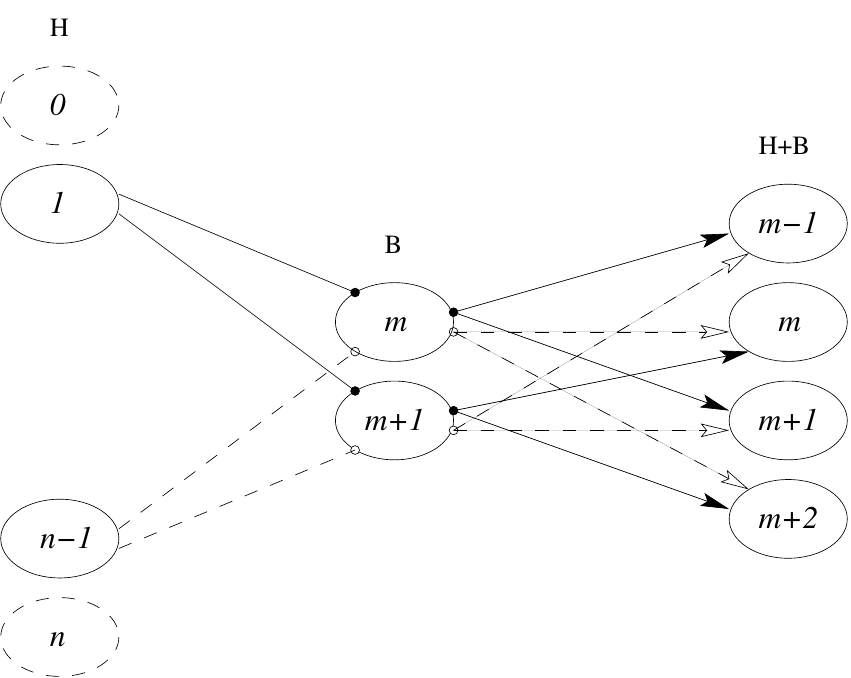}
\caption{Sumset $H+B$ for $n=2m+1$
and $B$ consisting of elements of weight $m$ and $m+1$.}
\label{FirstF}
\end{figure}

%%%%%%%%%%%%%%%%%%%%%%%%%%%%%%%%%
\subsection{Proof of Theorem \ref{FirstMain}} 
%%%%%%%%%%%%%%%%%%%%%%%%%%%%%%%%%

In Section \ref{SimExSec}, part 2), we removed all the elements of weight~2 and $(n-2)$. 
We will now remove them one after another.

Consider first the case when $n$ is odd.

Set $e_{ij}:=e_i+e_j$ and 
$\overline{e_{ij}}:=\overline{e_0}+e_{ij}$.
In order to eliminate two elements $e_{ij}$ and $\overline{e_{ij}}$ from  $H+B$,
we have to remove from $B$ the following set:
$$
H\cup
\left\{H+e_{ij}\right\}
\cup
\left\{H+\overline{e_{ij}}\right\}.
$$
Notice however, that $\left\{H+e_{ij}\right\}$ and $\left\{H+\overline{e_{ij}}\right\}$ coincide
and their intersection with $H$ is $\{e_i,e_j,\overline{e_{i}},\overline{e_{j}}\}$.
Thus the cardinality of the removed set is $4n-4$.

Finally, we fix $(\ell,k)$ with $\ell<k$ and let $B$ be the following set 
$$
\displaystyle
B=\Zn\setminus\Big(
H\cup\bigcup_{(i,j)\prec(\ell,k)}
\left\{H+e_{ij}\right\}\Big),
$$
where $\prec$ is the lexicographical order. 
It can be easily proved by induction that the removed 
set is of cardinality
$$
\big((k-1)(k-2)+2\ell+2\big)n -\frac{2}{3}\,(k-2)(k-1)k+2\,k\ell.
$$

The corresponding sumset is then of the form
$$
H+B=
\Zn\setminus
\{e_0,\,\overline{e_{0}},\,e_{ij}\,\overline{e_{ij}}\;,\;(i,j)\prec (\ell,k)\}.
$$
Its cardinality is
$$
\card(H+B)=2^n-(k-1)(k-2)-2\,\ell-2.
$$
We now apply Proposition  \ref{SumProp} to obtain identities (\ref{step}) for an odd $n$.

In the case $n$ is even, the proof is similar with $e_1$ playing the role of $\overline{e_{0}}$.

Theorem \ref{FirstMain} is proved.

\begin{figure}[hbtp]
\includegraphics[width=7.3cm]{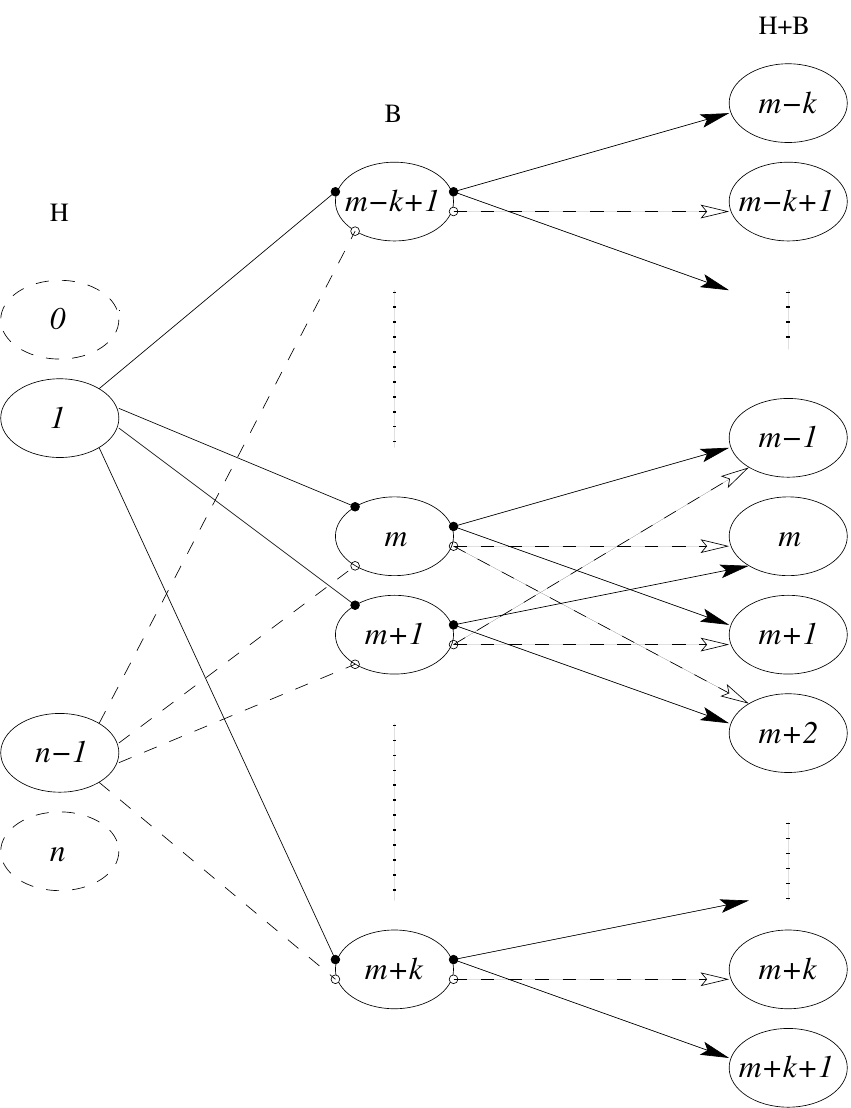}
%\end{center}
\caption{Sumset $H+B$ for $n=2m+1$
and $B$ consisting of elements of weight between $(m-k+1)$ and $(m+k)$.}
\label{5and6}
\end{figure}

%%%%%%%%%%%%%%%%%%%%%%%%%%%%%%%%%
\subsection{Proof of Theorem \ref{SecMain}} 
%%%%%%%%%%%%%%%%%%%%%%%%%%%%%%%%%

We write $n=2m+1$. 
Fix $k$ such that $0\leq k \leq m$ and set 
$$
B=\left\{x\;,\; |x|=m-k,\,m-k+1,\,\dots,\,m,\,m+1,\,\dots,\,m+k+1\right\}.
$$ 
The sumset is as follows (see Figure \ref{5and6})
$$
H+B=\left\{x\;,\; |x|=m-k-1,\,m-k,\,\dots,\,m+k+1,\,m+k+2\right\}.
$$
By Proposition \ref{SumProp}, this leads to the formulas of Theorem \ref{SecMain}.

Theorem \ref{SecMain} is proved.

%%%%%%%%%%%%%%%%%%%%%%%%%%%%%%%%%
\subsection{The twisting function $f_{\bbO}$ and representations of the Clifford algebras}
%%%%%%%%%%%%%%%%%%%%%%%%%%%%%%%%%

Let us finally mention that the twisting function $f_{\bbO}$ that defines the algebras $\bbO_n$
allows one to construct representations of the Clifford algebras $\Cl_{0,r}$ explicitly.

For  $x\in \Zn$, we define a $(2^n\times 2^n)$-matrix $G_x$ 
whose coefficients are indexed by elements $(y,t)\in \Zn\times \Zn$ and are given by
\begin{equation*}
G_x(y,t)=\left\{
\begin{array}{ll}
(-1)^{f_\bbO(y+t,\,y)}, & \text{if } y+t=x,\\[6pt]
0,& \text{otherwise}.
\end{array}
\right.
\end{equation*}
Note that for $x=(0,\dots,0)$ one obtains $G_0=\Id$.

The matrices $G_x$ satisfy the following properties.
\begin{enumerate}
\item[(i)]
$G_x^2=(-1)^{f_\bbO(x,x)}\,\Id,
\qquad
G^T_x=G_x^{-1}\;=\;(-1)^{f_\bbO(x,x)}\,G_x.$
\medskip

\item[(ii)]
$G_xG_{x'}=-G_{x'}G_x$
if and only if
$
f_{\bbO}(x,x')+f_{\bbO}(x',x)=1.
$
\end{enumerate}
This leads to representations of Clifford algebras in $\R[2^n]$.
Indeed, if $A$ is a subset of $r$ elements of $\Zn$ such that $f_\bbO(x,x)=1$ and 
$f_{\bbO}(x,x')+f_{\bbO}(x',x)=1$,
for all $x,x'\in A$, 
then the matrices $G_x$, $x\in A$,
represent the real Clifford algebra $\Cl_{0,r}$ in the algebra
$\R[2^n]$.

\begin{ex}
We describe explicitly the irreducible representations of 
some of the simple real Clifford algebras.
We give a possible choice of sets $A$. 
The condition on the set $A$ 
can be expressed using the generating function $\a_{\bbO}$. 
We use the notation of Section \ref{HurSecT}.
\begin{enumerate}
\item[\textbullet] If $n\equiv 3 \mod 4$, an irreducible representation of $\Cl_{0,2n}=\R[2^n] $ is given by
$$
G_x, \;x\in \{e_i, \overline{e_i}, \; 1\leq i \leq n\}.
$$
\item[\textbullet] If $n\equiv 1,3 \mod 4$, an irreducible real representation of $\Cl_{0,2n-1}=\C[2^{n-1}] $ is given by
$$
G_x,\; x\in \{ e_i, e_1+e_j, \;1\leq i \leq n,\; 1<j \leq n\}.
$$

\item[\textbullet] If $n\equiv 2, 3 \mod 4$, an irreducible real representation of $\Cl_{0,2n-2}=\bbH[2^{n-2}] $ is given by 
$$
G_x,\; x\in \{ e_i, e_1+e_j, \;1\leq i \leq n,\; 1<j < n\}.
$$
\end{enumerate}
\end{ex}

\medskip

\noindent \textbf{Acknowledgments}.
The second author has benefited from the award of a
\textit{Leibniz Fellowship}
at the Mathematisches Forschungsinstitut Oberwolfach (MFO).
The third author is also grateful to MFO for hospitality.
We would like to thank C. Duval for helpful comments.
We are also deeply grateful to the anonymous referee for
a number of helpful remarks and suggestions.

%%%%%%%%%%%%%%%%%

\end{document}